\documentclass[11pt,reqno,tbtags,a4paper]{amsart}
\usepackage{amssymb}
\usepackage{xpunctuate}
\usepackage{url}
\usepackage[square,numbers]{natbib}
\bibpunct[, ]{[}{]}{;}{n}{,}{,}

\title
{On the Gromov--Prohorov distance}

\date{27 May, 2020; revised 2 June, 2020}

\author{Svante Janson}
\thanks{Partly supported by the Knut and Alice Wallenberg Foundation}
\address{Department of Mathematics, Uppsala University, PO Box 480,
SE-751~06 Uppsala, Sweden}
\email{svante.janson@math.uu.se}
\newcommand\urladdrx[1]{{\urladdr{\def~{{\tiny$\sim$}}#1}}}
\urladdrx{http://www2.math.uu.se/~svante/}


\numberwithin{equation}{section}

\renewcommand\le{\leqslant}
\renewcommand\ge{\geqslant}

\allowdisplaybreaks

\theoremstyle{plain}
\newtheorem{theorem}{Theorem}[section]

\newtheorem{prop}[theorem]{Proposition}

\newtheorem{definition}[theorem]{Definition}
\theoremstyle{definition}

\newtheorem{exampleqqq}[theorem]{Example}

\newtheorem{remarkqqq}[theorem]{Remark}
\newenvironment{remark}{\begin{remarkqqq}}
  {\hfill\qedsymbol\end{remarkqqq}}

\theoremstyle{remark}

\newenvironment{romenumerate}[1][-10pt]{
\addtolength{\leftmargini}{#1}\begin{enumerate}
 \renewcommand{\labelenumi}{\textup{(\roman{enumi})}}%
 }{\end{enumerate}}

\newcounter{oldenumi}
{\setcounter{oldenumi}{\value{enumi}}
\begin{romenumerate} \setcounter{enumi}{\value{oldenumi}}}
{\end{romenumerate}}

\newcounter{thmenumerate}

\newcounter{steps}

\newcommand{\refT}[1]{Theorem~\ref{#1}}

\newcommand{\refR}[1]{Remark~\ref{#1}}

\newcommand{\refP}[1]{Proposition~\ref{#1}}

\newcommand{\refD}[1]{Definition~\ref{#1}}
\newcommand{\refDs}[1]{Definitions~\ref{#1}}

\begingroup
  \divide\count255 by 60
  \multiply\count255 by -60
  \advance\count255 by \time
  \ifnum \count255 < 10 \xdef\klockan{\the\count1.0\the\count255}
  \else\xdef\klockan{\the\count1.\the\count255}\fi
\endgroup

\newcommand\nopf{\qed}   

\newcommand\set[1]{\ensuremath{\{#1\}}}
\newcommand\bigset[1]{\ensuremath{\bigl\{#1\bigr\}}}

\newcommand\xpar[1]{(#1)}
\newcommand\bigpar[1]{\bigl(#1\bigr)}
\newcommand\Bigpar[1]{\Bigl(#1\Bigr)}

\newcommand\bigsqpar[1]{\bigl[#1\bigr]}

\newcommand\xcpar[1]{\{#1\}}

\newcommand\bigabs[1]{\bigl\lvert#1\bigr\rvert}

\def\rompar(#1){\textup(#1\textup)}    

\def\xexp(#1){e^{#1}}

\newcommand\ntoo{\ensuremath{{n\to\infty}}}

\newcommand\punkt{\xperiod}    
\newcommand\iid{i.i.d\punkt}    
\newcommand\ie{i.e\punkt}
\newcommand\eg{e.g\punkt}

\newcommand{\textas}{\text{a.s.}}
\newcommand{\aex}{a.e\punkt}

\newcommand{\tend}{\longrightarrow}
\newcommand\dto{\overset{\mathrm{d}}{\tend}}
\newcommand\pto{\overset{\mathrm{p}}{\tend}}

\newcommand\bbR{\mathbb R}

\newcounter{CC}
\newcounter{cc}

\newcommand\E{\operatorname{\mathbb E{}}}
\renewcommand\P{\operatorname{\mathbb P{}}}

\newcommand\diam{\operatorname{diam}}
\newcommand\supp{\operatorname{supp}}

\newcommand\gd{\delta}

\newcommand\gf{\varphi}

\newcommand\kk{\kappa}
\newcommand\gl{\lambda}
\newcommand\gL{\Lambda}

\newcommand\eps{\varepsilon}

\newcommand\cL{{\mathcal L}}
\newcommand\cM{\mathcal M}

\newcommand\cP{\mathcal P}

\newcommand\cX{{\mathcal X}}

\newcommand\tB{\widetilde B}

\newcommand\indic[1]{\boldsymbol1\xcpar{#1}}

\newcommand\qw{^{-1}}

\newcommand\oi{\ensuremath{[0,1]}}

\newcommand\ooo{[0,\infty)}

\newcommand\push{_\sharp}
\newcommand\nn{^{(n)}}

\newcommand\nnoo{\relax}
\newcommand\zoo{\relax}

\newcommand\GHP{Gromov--Hausdorff--Prohorov}
\newcommand\GH{Gromov--Hausdorff}
\newcommand\GP{Gromov--Prohorov}

\newcommand\Xoo{X_\infty}
\newcommand\dboxx[1]{\underline{\square}_{#1}}
\newcommand\dbox{\dboxx1}

\newcommand\dboxa{\dboxx{a}}
\newcommand\leb{\gl}
\newcommand\mpp{measure preserving}

\newcommand\dGP{d_{\mathsf {GP}}}
\newcommand\dGH{d_{\mathsf {GH}}}
\newcommand\dGHP{d_{\mathsf {GHP}}}
\newcommand\dPx[1]{d_{\mathsf{P,}#1}}
\newcommand\dPa{\dPx{a}}
\newcommand\dGPx[1]{d_{\mathsf{GP,}#1}}
\newcommand\dGPa{\dGPx{a}}
\newcommand\mms{metric measure space}
\newcommand\Gto{\overset{\mathrm{G}}{\tend}}

\begin{document}

\begin{abstract} 
We survey some basic results on the Gromov--Prohorov distance
between metric measure spaces.
(We do not claim any new results.)

We give several different definitions and show the equivalence of them.
We also show
that convergence in the Gromov--Prohorov distance is equivalent
to convergence in distribution of the array of 
distances between finite sets of random points.
\end{abstract}

\maketitle

\section{Introduction}\label{S:intro}

\citet{Gromov} introduced a notion of convergence for \mms{s}
 $X=(X,d,\mu)$, where $(X,d)$ is a
complete and separable metric space, and $\mu$ is a finite Borel measure on $X$.
We assume in the sequel that $\mu$ is a probability measure, \ie,
$\mu(X)=1$; the extension to arbitrary finite measures (as in \cite{Gromov})
is straightforward and left to the reader.

Gromov's convergence can be expressed in terms of a metric, 
known as the \emph{\GP} metric. In fact, there are several natural
definitions that are either completely equivalent, or equivalent within
(small) constant factors; these include Gromov's original definition 
of $\dboxa$ \cite[$3\frac12.B$]{Gromov}, and the version $\dGP$
by \citet[p.~762]{Villani}
and \citet{Greven++}
(\refDs{D1} and \ref{D3} below, respectively).

\citet{Gromov} also considered a different notion of convergence, based on
distances between random points in the space (\refD{D0} below),
and proved a convergence criterion \cite[p.~131]{Gromov}
relating this and convergence in his metric. In fact, these are equivalent
(\refT{TG}).

The purpose of the present note is to survey some different definitions and 
give proofs of the equivalence of them.
The results all are known, and we try to give original references, but there
might be unintentional omissions.

\begin{remark}
  The \GP{} distance $\dGP$ is closely related to the 
\emph{\GH{} distance} $\dGH$ 
(\cite[Chapter 7]{Burago}, \cite[Chapter 27]{Villani})
and the
\emph{\GHP{} distance} $\dGHP$
(\cite[p.~762]{Villani}, \cite[Section 6]{Miermont}).
Informally, convergence in the \GP{} distance means that there is ``almost''
a measure preserving isometry, but this may ignore parts of the spaces with 
zero or small measure;
convergence in the \GH{} distance does not involve measures at all, and means
that the spaces are almost isometric;
convergence in the \GHP{} distance combines both aspects.
\end{remark}

\begin{remark}
We consider throughout only complete separable metric spaces.
Several of the definitions and results extend to more general metric spaces,
but there are also serious technical problems in this case, and we prefer to
say no more about it.
\end{remark}

\section{Preliminaries}\label{Sprel}

\subsection{Some notation}\label{SSnot}

We denote Lebesgue measure on $\oi$ by  $\leb$, and
let $\oi $ denote the measure space $(\oi,m)$.

If $x\in X$, where $X$ is a metric space, and $r>0$,
then $B(x,r):=\set{y\in X:d(y,x)\le r}$ is the closed ball with centre
$x$ and radius $r$.

If $X$ is a  metric space, then $\cP(X)$ is the space
of all (Borel) probability measures on $X$.
We equip $\cP(X)$ with the standard topology of weak convergence; see \eg{}
\cite{Billingsley} or \cite{Bogachev}.

If $\mu\in\cP(X)$, then $\xi\sim\mu$ means that $\xi$ is a random variable
in $X$ with distribution $\mu$.
We use $\dto$ and $\pto$ for convergence in distribution and in probability,
respectively,
of random variables.

If $\mu\in\cP(X)$, then $\supp\mu$ denotes the \emph{support} of $\mu$,
\ie, the smallest closed subset of $X$ with full measure.
We have 
\begin{align}
\supp\mu=\bigset{x\in X:\mu(B(x,r))>0\quad \forall r>0}.  
\end{align}

If $X$ and $Y$ are metric spaces, $\gf:X\to Y$ is measurable,
and $\mu\in\cP(X)$, then $\gf\push(\mu)\in\cP(Y)$
denotes the \emph{push-forward} of $\mu$, defined by
\begin{align}
  f\push(A)= f\bigpar{\gf\qw(A)}
\end{align}
for any measurable $A\subseteq Y$.
Equivalently, if $\xi\sim\mu$, then $\gf\push(\mu)$ is the distribution of
$\gf(\xi)$ (which is a random variable in $Y$).

A measurable map $\gf:(X,\mu)\to (Y,\nu)$, where $(X,\mu)$ and $(Y,\nu)$ are
probability spaces, is \emph{\mpp} if $\gf\push(\mu)=\nu$.

\subsection{The Prohorov distance}\label{SSProhorov}
Let $X=(X,d)$ be a complete separable metric space.

If $B$ is a subset of $X$ and $\eps>0$, let
\begin{align}
  B^\eps:=\set{x: d(x,B) \le \eps}.
\end{align}

The \emph{Prohorov distance} 
$\dPa(\mu,\mu')$ (where $a>0$ is a parameter, usually chosen to be 1)
between two probability measures $\mu$ and
$\mu'$ in  $\cP(X)$ is defined as the infimum of $\eps>0$
such that, for every Borel set $B\subseteq X$, 
\begin{align}\label{dpr}
\mu'(B)\le \mu (B^\eps)+a\eps.
\end{align}
It is easily seen that this is symmetric in $\mu$ and $\mu'$, and that
\eqref{dpr} (for every $B$) implies also
\begin{align}\label{dpr'}
\mu(B)\le \mu' (B^\eps)+a\eps.
\end{align}

\begin{remark}
  Note that different choices of the parameter
$a$ yield distances that are equivalent within constant factors.
(We use $a$ only for greater flexibility and precision in the equivalences
below.) In fact, $\dPa$ equals $a\qw \dPx1$ evaluated in the 
metric space $(X,ad)$ with a rescaled metric.
\end{remark}

\begin{remark}\label{Rkyfan}
  The Prohorov distance  has also a dual formulation: $\dPa(X'X')$ equals the minimal
$\eps\ge0$ such that there exist two random variables $\xi\sim\mu$ 
and $\xi'\sim\mu'$ in $X$ such that
\begin{align}\label{kyfan}
  \P\bigpar{d(\xi,\xi')>\eps} \le a\eps.
\end{align}
See \cite[Corollary 7.5.2]{Rachev}.
\end{remark}

\begin{remark}
  The Prohorov distance is a metric on 
$\cP(X)$ that generates the weak topology
\cite[Appendix III]{Billingsley},
\cite[Theorem 8.3.2]{Bogachev}.
\end{remark}

See further \cite{Billingsley}, \cite{Bogachev}, \cite{Rachev} and the
survey \cite{SJN21}.


\section{The \GP{} distance}\label{SGP}

We give in this section several definitions of 
a (pseudo)distance between two 
(complete, separable) 
\mms{s}
$X=(X,d,\mu)$ and $X'=(X',d',\mu')$.
The definitions are all equivalent within constant factors,
and we can choose any of them as the definition 
of the \emph{\GP} distance $\dGP(X,X')$.
(Our default choice is $\dGP:=\dGPx1$.)

The original definition by 
\citet[Section $3\frac12$.3]{Gromov} can be written as follows.
Here $a>0$ is an arbitrary parameter; the distances $\dboxa$
for different values of
$a$ are obviously equivalent, and usually we choose $a=1$.

\begin{definition}\label{D1}
  $\dboxa(X,X')$ is the infimum of $\eps>0$ such that there exist \mpp{} maps 
$\gf:\oi\to X$ and $\gf':\oi\to X'$ and a set $W_\eps\subseteq\oi$ such that
\begin{align}\label{d1a}
&\leb(W_\eps)\le a\eps
\\
 &\bigabs{d\bigpar{\gf(x_1),\gf(x_2)}-d'\bigpar{\gf'(x_1),\gf'(x_2)}}\le \eps,
\quad x_1,x_2\in\oi\setminus W_\eps.
\label{d1}
\end{align}
\end{definition}

We give an alternative, equivalent, definition.
Recall that a \emph{coupling} of the measures $\mu$ on $X$ and $\mu'$ on
$X'$ is a probability measure $\nu$ on $X\times X'$ such that the marginals
are $\mu$ and $\mu'$. Recall also that a \emph{relation} between $X$ and
$X'$ is any subset $R\subseteq X\times X'$.

\begin{definition}\label{D2}
  $\dboxa(X,X')$ is the infimum of $\eps>0$ such that
there exist 
a Borel relation $R\subseteq X\times X'$ and
a coupling $\nu$ of $\mu$ and $\mu'$, 
such that
\begin{align}\label{p1a}
&\nu(R)\ge 1-a\eps,\\&
  (x_1,x_1'),(x_2,x'_2)\in R \implies
\bigabs{d(x_1,x_2) - d'(x_1',x'_2)} \le \eps.
\label{p1}
\end{align}
\end{definition}

It is easily seen that we may require the relation $R$ to be closed.

\begin{prop}\label{P1}
\refDs{D1} and \ref{D2} agree.
\end{prop}

\begin{proof}
  Given $\gf,\gf'$ and $W_\eps$ as in \refD{D1}, 
define 
\begin{align}\label{p1b}
R_0:=\bigset{\bigpar{\gf(x),\gf'(x)}:x\in \oi\setminus W_\eps}.  
\end{align}
Then \eqref{d1} shows that \eqref{p1} holds for $R_0$.
Let $R:=\overline{R_0}$; then \eqref{p1} holds by continuity.

Furthermore, let $\Phi:=(\gf,\gf'):\oi\to X\times X'$ and let
$\nu$ be the probability measure $\Phi\push(\leb)$ on $X\times X'$.
Then $\nu$ is a coupling of $\mu$ and $\mu'$, and
\begin{align}\label{p1c}
  \nu(R) = \leb\bigpar{\Phi\qw(R)}
\ge \leb\bigpar{\oi\setminus W_\eps}\ge 1-a\eps.
\end{align}
Hence, the conditions in \refD{D2} hold.

Conversely, suppose that $R$ and $\nu$ are as in \refD{D2}.
Then $\nu$ is a probability measure on the Polish space $X\times X'$, and
thus there exists a measure preserving map $\Phi:\oi\to (X\times X',\nu)$,
see \cite[Theorem 3.19 or Lemma 3.22]{Kallenberg}.
Write $\Phi=(\gf,\gf')$. Then, $\gf$ and $\gf'$ are \mpp{} maps $\oi\to X$
and $\oi\to X'$. 
Let $W_\eps:=\oi\setminus \Phi\qw(R)$. Then \eqref{d1} holds by \eqref{p1},
and
\begin{align}
  \leb(W_\eps) = 1-\leb(\Phi\qw(R))
= 1-\nu(R) \le a\eps.
\end{align}
Hence, $\gf,\gf'$ and $W_\eps$ are as in \refD{D1}.
\end{proof}

Another metric was defined
by \citet[p.~762]{Villani}
and \citet{Greven++}.

\begin{definition}\label{D3}
  $\dGPa(X,X')$ equals the infimum of $\eps>0$ such that  
there exists a metric space $Z$ with subspaces $Y,Y'\subseteq Z$
and isometries $\gf:X\to Y$ and $\gf':X'\to Y'$ such that the Prohorov
distance
\begin{align}\label{d3}
\dPx{a}\bigpar{\gf\push(\mu) ,\gf'\push(\mu')} \le \eps. 
\end{align}
\end{definition}
In other words, 
$\dGPa(X,X')$ is the infimum of the Prohorov distance 
between $\gf\push(\mu)$ and $\gf'\push(\mu')$ over all metric spaces $Z$ and
isometric embeddings $\gf:X\to Z$ and $\gf':X'\to Z$.

Note that we may assume that the metric space $Z$ in \refD{D3} is complete
and separable, since otherwise we may replace $Z$ by first its completion
and then the closure of $Y\cup Y'$ (or conversely).

\begin{prop}[{\citet{Lohr}}]\label{P2}
For any 
\mms{s} $X$ and $X'$ and any
$a>0$, 
\begin{align}\label{p2}
\dboxa(X,X') = 2\dGPx{2a}(X,X')
.\end{align}
\end{prop}

\begin{proof}
We argue as for
the corresponding result for the \GHP{} distance in \cite{Miermont};
see also \cite[Section 7.3]{Burago}.

  Let $\eps>\dboxa(X,X')$ and let $R$ and $\nu$ be as in \refD{D2}.
Let $Z:=X\sqcup X'$ be the disjoint union of $X$ and $X'$, 
and define a metric $\gd$ on $Z$ that equals $d$ on $X$, $d'$ on $X'$, and,
for $x\in X$ and $x'\in X'$,
\begin{align}
  \gd(x,x') := \inf\bigpar{d(x,y) + \eps/2+d'(y',x'): (y,y')\in R}.
\end{align}
It is easily verified that this really defines a metric, see 
\eg
\cite[Proof of Proposition 6]{Miermont}, and that $\gd(x,x')=\eps/2$
when $(x,x')\in R$.

Regard $X$ and $X'$ as subspaces of $Z$,
and 
let $(\xi,\xi')$ be a random variable in $X\times X'$ with distribution
$\nu$. 
If $(\xi,\xi')\in R$, then $\gd(\xi,\xi')=\eps/2$; hence, 
\begin{align}
  \P\bigpar{\gd(\xi,\xi')>\eps/2}
\le \P\bigpar{(\xi,\xi')\notin R}
= 1-\nu(R) \le a\eps
=2a\cdot\eps/2
.\end{align}
Hence, see \refR{Rkyfan}, $\dPx{2a}(\mu,\mu')\le \eps/2$.
Consequently, by \refD{D3},
\begin{align}
\dGPx{2a}(X,X')\le \dPx{2a}(\mu,\mu')\le \eps/2
.\end{align}

Conversely, suppose that 
$\dGPx{2a}(X,X')\le \eps$.
Then
there exist $Y, Y'$ and $\gf,\gf'$ as in 
\refD{D3}, with $a$ replaced by $2a$.
We may assume that $X=Y$ and $X'=Y'$.
Thus,
\begin{align}
\dPx{2a}(\mu,\mu')\le \eps.  
\end{align}
By \refR{Rkyfan}
there exist random variables 
$\xi$ and $\xi'$ in $Z$ such that
\begin{align}\label{kyfan2}
  \P\bigsqpar{d(\xi,\xi')>\eps} \le 2a\eps.
\end{align}
Let $R:=\bigset{(x,x')\in X\times X':d(x,x')\le \eps}$.
This is a closed relation, 
and it follows from
\eqref{kyfan2} that
if $\nu$ is the distribution of $(\xi,\xi')$, then
\begin{align}
  \nu(R)=\P\bigsqpar{(\xi,\xi')\in R}
=\P\bigsqpar{d(\xi,\xi')\le\eps}
\ge 1-2a\eps.
\end{align}
Furthermore, \eqref{p1} holds with $\eps$ replaced by $2\eps$.
Hence, \refD{D2} shows that
$\dboxa(X,X')\le 2\eps$.
\end{proof}

\begin{remark}
Definitions \ref{D2} and \ref{D3} 
are analogues of similar definitions in \cite{Burago} and
\cite{Miermont} for the related \GH{} and \GHP{} distances.
In particular,  they 
correspond to the
definition and Proposition 6 in \cite[Section 6.2]{Miermont} if we ignore
the Hausdorff part; note that the only significant difference between
the conditions in \cite[Proposition 6]{Miermont} and in \refD{D2} above
is that in \cite{Miermont} (for $\dGHP$), the coupling $R$ is supposed to 
be a \emph{correspondence}, \ie, the projections of $R\to X$ and $R\to X'$
are onto.
(In other words, every $x\in R$ is related to some $x'\in X'$, and conversely.)
\end{remark}

\begin{remark}\label{R0}
It is easy to see, perhaps simplest from \refD{D2}, that the triangle
inequality holds for $\dboxa$ and $\dGPa$; hence the distances above are
pseudometrics. 
Note that $\dGP(X,X')=0$ may hold not only for isomorphic $X$ and $X'$ (in
the obvious sense that there exists a \mpp{} bijection).
In fact, for any $X=(X,\mu)$, 
if we let $X':=\supp\mu$,
then
\begin{align}
\dGP\bigpar{(X,\mu), (X',\mu)}=0.  
\end{align}
We will see in \refT{TG1} that this is essentially the only way that $\dGP$
fails to be a metric.
\end{remark}

We note two basic results by \citet{Gromov}, to which we refer for proofs.

\begin{theorem}[{\citet[Corollary in $3\frac12.6$]{Gromov}}]\label{TG1}
If $X=(X,\mu)$ and $(X',\mu')$ are \mms{s}, then
  $\dGP(X,X')=0$ if and only if 
$(\supp\mu,\mu)$ and $(\supp\mu',\mu')$ are isomorphic 
\mms.

In other words, $\dGP$ is a metric on the set $\cX$ of equivalence classes
(under isomorphism) of \mms{s} $(X,\mu)$
with full support, $\supp \mu=X$.
\nopf
\end{theorem}

\begin{theorem}[{\citet[Corollary in $3\frac12.12$]{Gromov}}]\label{TG2}
The metric space $(\cX,\dGP)$ is complete and separable.
\nopf
\end{theorem}

\section{Convergence}\label{Sconv}

\citet{Gromov} considered also convergence of metric spaces in terms of 
arrays of distances between points in the following way (in our notation).

For an integer $\ell\ge1$, let $\cM_\ell$ be the space of real $\ell\times
\ell$ matrices; 
note that $\cM_\ell=\bbR^{\ell^2}$ is a complete separable metric space.

For a metric space $X=(X,d)$ and $\ell\ge1$, let
$\rho_\ell:X^\ell\to \cM_\ell$ be the map given by the entries
\begin{align}\label{rhor}
  \rho_\ell(x_1,\dots,x_\ell)_{ij}=
  \rho_\ell(x_1,\dots,x_\ell;X,d)_{ij}:=
    d(x_i,x_j).
\end{align}
If $X=(X,d,\mu)$ is a \mms,
define for $\ell\ge1$, 
the  measure
\begin{align}\label{taur}
  \tau_\ell(X)=\tau_\ell(X,d,\mu):={\rho_\ell}\push\xpar{\mu^\ell}\in \cP(\cM_\ell),
\end{align}
the push-forward of the measure
$\mu^\ell\in \cP(X^\ell)$ along $\rho_\ell$. 
In our setting with a probability measure $\mu$,
we can, equivalently, define $\tau_\ell$ by
letting $\xi_1,\dots,\xi_\ell$ be
\iid{} (independent, identically distributed)
 random points in $X$ with $\xi_i\sim\mu$;
then
\begin{align}\label{taurus}
  \tau_\ell(X):=\cL\bigpar{\rho_\ell(\xi_1,\dots,\xi_\ell;X)}, 
\end{align}
the distribution of  the random matrix 
$\rho_\ell(\xi_1,\dots,\xi_\ell)\in \cM_\ell$.

We then define convergence of a sequence of \mms{} as follows.
(All unspecified limits below are as \ntoo.)

\begin{definition}\label{D0}
Let $(X_n)_1^\infty$ and $X$ be \mms{s}.
We say that $X_n\Gto X$ if
for every $\ell\ge1$,
\begin{align}\label{d0}
\tau_\ell(X_n)\to \tau_\ell(X)  
\qquad\text{in }\cP(\cM_\ell)  
.\end{align}
\end{definition}
By \eqref{taurus}, the condition \eqref{d0} can also be written
\begin{align}\label{d0d}
\rho_\ell(\xi\nn_1,\dots,\xi\nn_\ell;X_n)
\dto  
\rho_\ell(\xi_1,\dots,\xi_\ell;X) , 
\end{align}
where $(\xi\nn_i)$ are \iid{} random points in $X_n$ with
$\xi\nn_i\sim\mu_n$.

In fact, as stated in the next theorem, convergence in this sense 
is equivalent to convergence
in the \GP{} distance. 

\begin{theorem}[{\citet[Theorem 5]{Greven++}}]  \label{TG}
Let $(X_n)_1^\infty$ and $X$ be \mms{s}.
Then $X_n\Gto X$ if and only if
$\dGP(X_n,X)\to0$.
\end{theorem}

\begin{remark}
  We use the notation $\Gto$ in honour of Gromov,
since the property \eqref{d0} is studied
in \cite{Gromov}; see \eg{} \cite[$3\frac12$.14]{Gromov},
which discusses the relation with convergence in the \GP{} distance.
However (as far as we know), \refT{TG} is not stated explicitly in
\cite{Gromov}.
(The easy implication $\impliedby$ is implicit in 
\cite[$3\frac12$.6]{Gromov}; the converse is almost, but not quite, in
\cite[$3\frac12$.14]{Gromov}.)
\end{remark}

\begin{remark}\label{R=}
\citet[$3\frac12$.5]{Gromov}
proved the far from obvious fact that if $X$ and $Y$
are two \mms{s} such that the measures have full support, then
\begin{align}
  \label{r=}
\tau_\ell(X)=\tau_\ell(Y), 
\qquad \forall \ell\ge1
\end{align}
if and only if $X$ and $Y$ are isomorphic.

Equivalently, for any \mms{s} $X$ and $Y$, 
\eqref{r=} holds if and only if $\dGP(X,Y)=0$.
(Cf.\ \refT{TG1}, which is proved in \cite{Gromov} using this fact.)
\end{remark}

\begin{remark}\label{Rbad}
  As remarked by \citet[$3\frac12$.14]{Gromov},
if we instead of \eqref{d0} just assume that 
\begin{align}\label{dbad}
  \tau_\ell(X_n)\to \nu_\ell,
\qquad \ell\ge1,
\end{align}
for some probability measures $\nu_\ell\in\cP(\cM_\ell)$,
then $X_n$ does not have to converge, \ie, \eqref{dbad} does not
imply that the limits $\nu_\ell=\tau_\ell(X)$ for some \mms{} $X$.
For example \cite[$3\frac12$.14 and $3\frac12$.18]{Gromov},
if $X_n$ is the unit sphere $S^n$ with normalized surface measure and, say,
the intrinsic (Riemannian) metric $d_n$, 
and $(\xi\nn_i)_i$ are \iid{} uniformly random points in $X_n$,
then,
\begin{align}
  d_n\bigpar{\xi\nn_i,\xi\nn_j}\pto \pi/2,
\end{align}
for any distinct $i$ and $j$, and thus
\eqref{dbad} holds with $\nu_\ell$ the point mass at the matrix 
$\frac{\pi}{2}\bigpar{\indic{i\neq j}_{i,j=1}^\ell}$.
However, there is no \mms{} $(X,d,\mu)$ with 
$\tau_\ell(X)=\nu_\ell$, which would mean that if $\xi_1,\xi_2$ are \iid{} random
points in $X$ with $\xi_i\sim\mu$, then $d(\xi_1,\xi_2)=\pi/2$ a.s.
(This would imply that for any $r<\pi/2$ and $\mu$-\aex{} $x_1\in X$,
$\mu\bigpar{B(x,r)}=0$, and thus $x\notin\supp\mu$; hence $\mu(\supp\mu)=0$,
a contradiction.)
\end{remark}

\begin{remark}\label{Roo}
  We have (implicitly) assumed above that $\ell$ is a finite integer.
However, we can also  use the same definitions \eqref{rhor}--\eqref{taurus}
for $\ell=\infty$, noting that $M_\infty=\bbR^{\infty^2}$ still is a Polish
space, \ie, it can be regarded as a complete separable metric space.
(The choice of metric is of no importance to us.)

It is easy to see that the condition \eqref{d0}, or equivalently
\eqref{d0d},
for every finite $\ell$ is equivalent to the same condition for
$\ell=\infty$.
Hence, $X_n\Gto X$ can also be defined by
$\tau_\infty(X_n)\to\tau_\infty(X)$ in $\cP(\cM_\infty)$, or by
\begin{align}\label{dood}
\rho_\infty(\xi\nn_1,\xi\nn_2,\dots;X_n)
\dto  
\rho_\infty(\xi_1,\xi_2,\dots;X) 
.\end{align}
\end{remark}

\begin{proof}[Proof of \refT{TG}]
Let $X_n=(X_n,d_n,\mu_n)$ and $X=(X,d,\mu)$.
As above, 
let $(\xi\nn_i)_i$ be \iid{} random points in $X_n$ with 
$\xi\nn_i\sim \mu_n$, and
let $(\xi_i)_i$ be \iid{} random  points in $X$ with 
$\xi_i\sim \mu$.
(We may also write $\xi\nn$ and $\xi$
without index when the index does not matter.)

First, suppose that $\dGP(X_n,X)\to0$. By \refP{P2}, 
then $\dbox(X_n,X)\to0$, 
and thus 
there exists a sequence $\eps_n\to0$
such that $\dbox(X_n,X)<\eps_n$ and hence,
see \refD{D2}, there exists a coupling $\nu_n$ of $\mu_n$ and $\mu$
and a Borel relation $R_n\subseteq X_n\times X$ such that
\eqref{p1a}--\eqref{p1} hold for $\nu_n$,  $R_n$ and $\eps_n$ (with $d_n$
and $d$).
We may assume that each pair $(\xi\nn_i,\xi_i)$ has the distribution
$\nu$ on $X_n\times X$; thus 
\begin{align}\label{tja}
\P\bigpar{(\xi\nn_i,\xi_i)\in R_n}\ge 1-\eps_n  
\end{align}
by \eqref{p1a}.
Together with \eqref{p1} and the definition \eqref{rhor}, this implies
\begin{align}
&  \P\Bigpar{\bigabs{\rho_\ell(\xi\nn_1,\dots,\xi\nn_\ell;X_n)
-
\rho_\ell(\xi_1,\dots,\xi_\ell;X)} \le \ell^2\eps_n}
\notag\\&\qquad
\ge
\P\bigpar{(\xi\nn_i,\xi_i)\in R_n, 
i=1,\dots,\ell}
\ge 1-\ell\eps_n  
.\end{align}
This implies easily \eqref{d0d} for each $\ell$, and thus $X_n\Gto X$.

Conversely, suppose that $X_n\Gto X$, so that \eqref{d0d} holds.
Fix $r>0$, let $h(t):=(1-t/r)_+$ for $t\ge0$, and define
$g_n:X_n\to\ooo$ by
\begin{align}
  \label{gn1}
g_n(x):=\E h\bigpar{d_n(x,\xi\nn)}, 
\qquad n\ge1,
\end{align}
and similarly $g:X\to\ooo$ by $g(x):=\E h\bigpar{d(x,\xi)}$ 
Then $0\le h\le 1$ and $h\bigpar{d_n(x,y)}=0$ unless $y\in B(x,r)$;
hence
\begin{align}\label{gnu}
0\le  g_n(x) \le \mu_n\bigpar{B(x,r)}.
\end{align}

For any $m\ge1$, we have
\begin{align}\label{egon}
  g_n(x)^m = \E \prod_{i=1}^mh\bigpar{d_n(x,\xi\nn_i)}
\end{align}
and thus, if we define $H:\cM_{m+1}\to\bbR$ by 
$H\bigpar{(a_{ij})_{i,j}}:=\prod_{i=1}^mh(a_{m+1,i})$,
\begin{align}\label{eldn}
\E \bigsqpar{ g_n(\xi\nn)^m }&
= \E \prod_{i=1}^mh\bigpar{d_n(\xi\nn_{m+1},\xi\nn_i)}
\notag\\&
=\E H\bigpar{\rho_{m+1}(\xi\nn_1,\dots,\xi\nn_{m+1};X_n)}.
\end{align}
Similarly, 
\begin{align}\label{eld}
\E \bigsqpar{ g(\xi)^m }&
=\E H\bigpar{\rho_{m+1}(\xi_1,\dots,\xi_{m+1};X)}.
\end{align}
Note that $H$ is a bounded continuous function on $\cM_{m+1}$.
Consequently, the assumption $X_n\Gto \Xoo$ implies
by \eqref{d0d} and \eqref{eldn}--\eqref{eld}
\begin{align}
  \E [g_n(\xi\nn)^m]
\to
  \E [g\zoo(\xi\nnoo)^m],\qquad m\ge1.
\end{align}
 Thus, by the method of moments (recalling that $g(\xi)$ is bounded by
\eqref{gnu}, and thus its distribution is determined by its moments)
\begin{align}\label{gnto}
  g_n(\xi\nn)\dto g\zoo(\xi\nnoo).
\end{align}

If $x\in\supp\mu$, then $\P\bigpar{d(x,\xi)\le r/2}=\mu\bigpar{B(x,r/2)}>0$
and thus $g(x)>0$. Hence, $g(x)>0$ $\mu$-a.e., \ie, 
\begin{align}\label{g+}
g(\xi)>0 \qquad\textas
\end{align}
Fix $\eps>0$. By \eqref{g+}, there exists $\kk>0$ such that
$\P\bigpar{g(\xi)\le\kk} <\eps$.
Then, by \eqref{gnto}, there exists $n_0$ such that if $n\ge n_0$, then
\begin{align}\label{gneps}
  \P\bigpar{g_n(\xi\nn)\le\kk} <\eps.
\end{align}
Consider only $n\ge n_0$, and let
\begin{align}\label{An}
  A_n:=\bigset{x\in X_n:\mu_n\bigpar{B(x,r)}\ge\kk}.
\end{align}
By \eqref{gnu} and \eqref{gneps}, we have
\begin{align}\label{Ane}
  \mu_n(A_n)
=\P\bigpar{\mu_n(B(\xi\nn,r))\ge\kk}
\ge\P\bigpar{g_n(\xi\nn)\ge\kk}
>1-\eps.
\end{align}
Pick recursively points $x_{n1},x_{n2},\dots, x_{iN}$ in $A_n$ such that the balls
$B_{ni}:=B(x_{ni},r)$ are disjoint, and stop when this is no longer
possible. Since $\mu_n(B_{ni})\ge\kk$ for every $i$ by the definition of
$A_n$, this process has to stop at some $N=N_n\le 1/\kk$.

If $x\in A_n$, then $B(x,r)$ has to intersect some $B_{ni}=B(x_{ni},r)$, and 
thus $x\in B(x_{ni},2r)$. Consequently,
$A_n$ is covered by the $N$ balls $\tB_{ni}:=B(x_{ni},2r)$.
Hence, by \eqref{Ane},
\begin{align}\label{mixmax}
  \mu_n\Bigpar{\bigcup_{i=1}^N \tB_{ni}} \ge \mu_n(A_n)>1-\eps.
\end{align}

Furthermore, since $X_n\Gto X$, and thus by \eqref{d0d}
$\rho_2(\xi\nn_1,\xi\nn_2;X_n)\dto\rho_2(\xi_1,\xi_2;X)$,
we have $d_n(\xi\nn_1,\xi\nn_2)\dto d(\xi_1,\xi_2)$, and thus
the sequence $d_n(\xi\nn_1,\xi\nn_2)$ of random variables is tight
\cite[Lemma 4.8]{Kallenberg}.
Hence, there exists $D<\infty$ such that for all $n$,
\begin{align}\label{ada}
  \P\bigpar{d_n(\xi\nn_1,\xi\nn_2)>D}<\kk^2.
\end{align}
Suppose now that $x,y\in A_n$ and $d_n(x,y)>D+2r$. 
If
$\xi\nn_1\in B(x,r)$ and $\xi\nn_2\in B(y,r)$, then
$d_n(\xi\nn_1,\xi\nn_2)\ge d_n(x,y)-2r>D$;
consequently, using the independence of $\xi\nn_1$ and $\xi\nn_2$
together with the definition \eqref{An} of $A_n$,
\begin{align}
    \P\bigpar{d_n(\xi\nn_1,\xi\nn_2)>D}
\ge
   \P\bigpar{\xi\nn_1\in B(x,r)} \P\bigpar{\xi\nn_2\in B(y,r)}
\ge\kk^2,
\end{align}
which contradicts \eqref{ada}. Consequently,
$d_n(x,y)\le D+2r$ whenever $x,y\in A_n$, \ie,
\begin{align}
  \diam(A_n) \le D+2r
\end{align}
and thus
\begin{align}\label{sw}
  \diam \Bigpar{\bigcup_{i=1}^N \tB_{ni} }\le D+6r.
\end{align}

We have shown that for each positive $\eps$ and $r$, there exists $n_0$, 
$N_0$ ($=1/\kk^2$) and $D_1$ ($=D+6r$) 
such that for each $n\ge n_0$, there exists a collection
$\set{\tB_{ni}}_{i=1}^{N_n}$ of subsets $\tB_{ni}\subseteq X_n$
such that 
\begin{align}
N_n&\le N_0, 
\\
\diam\bigpar{\tB_{ni} }&\le 4r 
\\
  \diam \Bigpar{\bigcup_{i=1}^N \tB_{ni}} &\le D_1,
\\
  \mu_n\Bigpar{X_n\setminus \bigcup_{i=1}^N \tB_{ni}}&<\eps.
\end{align}
Furthermore, by increasing $N_0$ and $D_1$ if necessary,
this holds also for each $n<n_0$, as an easy consequence of the fact that each 
$\mu_n$ 
is a tight measure
(as is every probability measure in a Polish space
\cite[Theorem 1.4]{Billingsley}).

This shows that the sequence $(X_n)$ satisfies condition III in the
corollary on p.~131--132 in \citet{Gromov}; 
since we also assume \eqref{d0}, 
this corollary shows that
$X_n$ converges to some \mms{} $Y$
in $\dbox$, or equivalently in $\dGP$.
Furthermore (as in the proof of this corollary in \cite{Gromov}),
$\dGP(X_n,Y)\to0$ implies $\tau_\ell(X_n)\to\tau_\ell(Y)$
by the first part of the proof, and thus
$\tau_\ell(Y)=\tau_\ell(X)$ for every $\ell\ge1$.
By \refR{R=}, this implies $\dGP(X,Y)=0$, and thus also
$\dGP(X_n,X)\to0$.
\end{proof}

\section*{Acknowledgement}
I thank
Gabriel Berzunza Ojeda for help with references.

\newcommand\AAP{\emph{Adv. Appl. Probab.} }
\newcommand\JAP{\emph{J. Appl. Probab.} }
\newcommand\JAMS{\emph{J. \AMS} }
\newcommand\MAMS{\emph{Memoirs \AMS} }
\newcommand\PAMS{\emph{Proc. \AMS} }
\newcommand\TAMS{\emph{Trans. \AMS} }
\newcommand\AnnMS{\emph{Ann. Math. Statist.} }
\newcommand\AnnPr{\emph{Ann. Probab.} }
\newcommand\CPC{\emph{Combin. Probab. Comput.} }
\newcommand\JMAA{\emph{J. Math. Anal. Appl.} }
\newcommand\RSA{\emph{Random Structures Algorithms} }
\newcommand\DMTCS{\jour{Discr. Math. Theor. Comput. Sci.} }

\newcommand\AMS{Amer. Math. Soc.}
\newcommand\Springer{Springer-Verlag}
\newcommand\Wiley{Wiley}

\newcommand\vol{\textbf}
\newcommand\jour{\emph}
\newcommand\book{\emph}
\newcommand\inbook{\emph}
\def\no#1#2,{\unskip#2, no. #1,} 
\newcommand\toappear{\unskip, to appear}

\newcommand\arxiv[1]{\texttt{arXiv:#1}}
\newcommand\arXiv{\arxiv}

\def\nobibitem#1\par{}

\end{document}